\documentclass[a4paper, 12pt]{article} 

\usepackage[utf8]{inputenc}
\usepackage{lmodern}
\usepackage{geometry}
\usepackage{amssymb, amsfonts, amsthm, amsmath, mathtools}
\usepackage[english]{babel}
\usepackage[colorlinks]{hyperref}
\usepackage{tikz}

\theoremstyle{plain}
\newtheorem{theorem}{Theorem}[section]

\newtheorem{lemma}[theorem]{Lemma}
\newtheorem{proposition}[theorem]{Proposition}

\theoremstyle{definition}

\theoremstyle{remark}
\newtheorem{remark}[theorem]{Remark}

\newcommand{\N}{\mathbb{N}}
\newcommand{\Z}{\mathbb{Z}}
\newcommand{\R}{\mathbb{R}}

\newcommand{\ind}[1]{\mathbf{1}_{\left\{#1\right\}}}

\newcommand{\floor}[1]{{\left\lfloor #1 \right\rfloor}}

\numberwithin{equation}{section}

\DeclareMathOperator{\E}{\mathbb{E}}

\renewcommand{\P}{\mathbb{P}}

\newcommand{\dd}{\mathrm{d}}

\renewcommand{\bar}[1]{\overline{#1}}
\newcommand{\egaldistr}{\ {\overset{(d)}{=}}\ }

\renewcommand{\rho}{\varrho}
\renewcommand{\epsilon}{\varepsilon}

\title{Total number of births on the negative half-line of the binary branching Brownian motion in the boundary case}
\author{Xinxin Chen\thanks{\texttt{xchen@math.univ-lyon1.fr}, Institut Camille Jordan - C.N.R.S. UMR 5208 - Université Claude Bernard Lyon 1 (France).} \and Bastien Mallein\thanks{\texttt{mallein@math.univ-paris13.fr}, Université Sorbonne Paris Nord, LAGA, UMR 7539, F-93430, Villetaneuse, France}}
\date{\today}

\begin{document}

\maketitle

\begin{abstract}
The binary branching Brownian motion in the boundary case is a particle system on the real line behaving as follows. It starts with a unique particle positioned at the origin at time $0$. The particle moves according to a Brownian motion with drift $\mu = 2$ and diffusion coefficient $\sigma^2 = 2$, until an independent exponential time of parameter $1$. At that time, the particle dies giving birth to two children who then start independent copies of the same process from their birth place. It is well-known that in this system, the cloud of particles eventually drifts to $\infty$. The aim of this note is to provide a precise estimate for the total number of particles that were born on the negative half-line, investigating in particular the tail decay of this random variable.
\end{abstract}

\section{Introduction}

A \emph{branching Brownian motion} is a continuous-time particle system on the real line in which particles move according to independent Brownian motions and split at independent exponential times into children. These children then start independent copies of the branching Brownian motion from their birth place. In this article, we take interest in a binary branching Brownian motion, meaning that at each branching event, every particle splits into two daughter particles independently. We also assume the branching Brownian motion to be in the so-called \emph{boundary case} (following \cite{BiK04}), i.e.\@ that the Brownian motions driving the motion of the particles have drift $\mu =2$ and diffusion coefficient $\sigma^2 = 2$.

The branching Brownian motion can be constructed as a process decorating the infinite binary tree $\mathcal{U} := \cup_{n \in \Z_+} \{1,2\}^n$ following the classical Ulam-Harris notation, with the convention $\{1,2\}^0 = \{\emptyset\}$. For each $u \in \mathcal{U}$, we write $b_u$ and $d_u$ the birth- and death- times of the particle $u$, and for all $s \leq d_u$ we denote by $X_s(u)$ the position at time $s$ of the particle $u$ or the position of its ancestor alive at that time.

For all $t \geq 0$, let $\mathcal{N}_t =\{ u \in \mathcal{U} : b_u \leq t < d_u\}$ be the set of particles alive at time $t$. It is well-known that a branching Brownian motion in the boundary case satisfies local extinction and global survival properties. In other words, while $\mathcal{N}_t$ is almost surely non-empty for all $t \geq 0$, we have $\lim_{t \to \infty} \#\{ u \in \mathcal{N}_t : X_t(u) \in K \} = 0$ a.s.\@ for all compact set $K$.
More precisely, Bramson \cite{Bra78} obtained the precise asymptotic behaviour of the minimal position $M_t = \min_{u \in \mathcal{N}_t} X_t(u)$ occupied by a particle at time $t$, showing that
\begin{equation}
  \label{eqn:maxDep}
  M_t = \frac{3}{2} \log t + O_\P(1),
\end{equation}
with $O_\P(1)$ representing a tight family of random variables. Hence, for all $x \in \R$, after some finite time, there will be no particle in the interval $(-\infty,x)$.

The aim of this article is to study the law of the number $N$ of birth (or death) events occurring on the negative half-line, defined as
\begin{equation}
  \label{eqn:defineN}
  N := \sum_{u \in \mathcal{U}} \ind{X_u(d_u) \leq 0}.
\end{equation}
Precisely, we take interest in the right tail of the distribution of $N$, and we show that $\P(N \geq n) \sim \frac{1}{n}$ as $n \to \infty$.

More generally, for all $x \in \R$, we denote by $N_x$ the total number of birth event occurring below the level $x$ (with $N = N_0$), that can be written as
\begin{equation}
  \label{eqn:defN}
  N_x := \sum_{u \in \mathcal{U}} \ind{X_{d_u}(u) \leq x}.
\end{equation}
Remark that the random variable $N_x$ is related to, but different of, the number $\bar{N}_x$ of births that occurred in the branching Brownian motion with absorption at level $x$, defined as
\begin{equation}
  \label{eqn:barN}
  \bar{N}_x = \sum_{u \in \mathcal{U}} \ind{X_{d(u)}(u) \leq x, \forall s \leq d_u, X_s(u) \leq x}.
\end{equation}
The quantity $\bar{N}_x$ was introduced and studied by Kesten \cite{Kes78}, who proved it to be a.s.\@ finite if and only if the drift $\mu$ of the underlying Brownian motion is larger or equal to $2$. Increasingly tight estimates were obtained on $\bar{N}$ both in the boundary and the non-boundary cases \cite{ABB,Aid,Mai13,ahz13,BBHM15}.

The process $(\bar{N}_x, x \geq 0)$ is a Markovian branching process, at least as long as the number of children created in a branching event is non-random. In that case $\bar{N}_x$ is in one-to-one correspondence with the number $Z_x$ of individuals that hit level $x$ for the first time\footnote{And $(Z_x, x \geq 0)$ is Markovian, as it can be seen by applying the branching property along a stopping line, see next section.}. Conversely, the process $(N_x, x \geq 0)$ does not satisfy the Markov property, as particles that went above level $x$ for some time, then back below that level and gave birth are taken into account. However, it is possible to link the values of $N$ with $\bar{N}$ in such a way that the known tail of $\bar{N}_x$ helps us compute the tail of $N$, see Lemma~\ref{lem:stoppingLine} below. The main result of the article is the following.

\begin{theorem}
\label{thm:mainboundary}
Let $X$ be a branching Brownian motion in the boundary case. Writing $c = \log 2 + \gamma \approx 1.27036$, where $\gamma$ is the Euler-Mascheroni constant, we have
\[
  \E (N \ind{N \leq n}) = \log n + c + o(1) \quad \text{as $n \to \infty$}.
\]
It entails in particular $\P(N \geq n) \sim 1/n$ as $n \to \infty$.
\end{theorem}

As a comparison, the estimate of Maillard \cite[Theorem~1.1]{Mai13} on $\bar{N}$ can be written in this context: for all $x > 0$,
\[
  \P(\bar{N}_x > n) \sim_{n \to \infty} \frac{x e^{x}}{n (\log n)^2}.
\]
Therefore, the tail of $N$ is slightly heavier than the tail of $\bar{N}$, which indicates that a non-trivial contribution to the tail of $N$ comes from particles that cross $0$ at least one time before giving birth to descendants on the negative half-line.

\begin{remark}
\label{rem:taub}
This theorem is equivalent to
\[
  \E\left( e^{- \lambda N} \right) = 1 +\lambda \log \lambda + \left( 1 + c - \gamma \right)\lambda + o(\lambda), \textrm{ as } \lambda \downarrow 0.
\]
In other words, the asymptotic behavior of the Laplace transform of $N$ as $\lambda \to 0$ is linked to the asymptotic behavior of $\P(N \geq n)$ as $n \to \infty$. We refer to \cite[Lemma~8.3]{BIM20+} for a proof of that equivalence.
\end{remark}

\begin{remark}
One could obtain an estimate similar to Theorem~\ref{thm:mainboundary} for a branching Brownian motion with drift $\mu > 2$. In this situation, $N$ becomes integrable, but using the decomposition in Lemma~\ref{lem:stoppingLine} and straightforward adaptation of our arguments, one can obtain
\[
  \P(N > n) \sim_{n \to \infty} c n^{-\kappa},
\]
where $c > 0$ and $\kappa = \frac{\mu + \sqrt{\mu^2-2}}{\mu - \sqrt{\mu^2-2}}>1$.
\end{remark}

In the next section, we recall some useful estimates related to the branching Brownian motion. We then prove Theorem~\ref{thm:mainboundary} in Section~\ref{sec:stoppingLine}, by comparing the asymptotic behaviours as $x \to \infty$ of $N_x$ and $\bar{N}_x$.

\section{Stopping lines, branching random walk and the many-to-one lemma}
\label{sec:backgroud}

We begin by introducing the derivative martingale of the branching Brownian motion, defined as
$\displaystyle
  D_t := \sum_{u \in \mathcal{N}_t} X_t(u) e^{-X_t(u)}.
$
Lalley and Sellke \cite{LS87} proved that the derivative martingale converges a.s.\@ towards a non-degenerate limit
\begin{equation}
  \label{eqn:dinfity}
  D_\infty := \lim_{t \to \infty} D_t,
\end{equation}
which is a.s.\@ positive.

We then introduce \emph{optional stopping lines}. Stopping line techniques were pioneered in \cite{Cha,Jag}. Informally speaking, a stopping line is generalization of stopping time in the context of branching processes, such that different particles are stopped at different times. We take in particular interest in the following family of very simple cutting stopping lines
\begin{equation}
  \label{eqn:linex}
  \mathcal{L}_x := \begin{cases}
    \left\{ (u,t) \in \mathcal{U} \times \R_+ : b_u \leq t < d_u, X_t(u) = x, \forall s < t, X_s(u) < x \right\},& x \geq 0,\\
    \left\{ (u,t) \in \mathcal{U} \times \R_+ : b_u \leq t < d_u, X_t(u) = x, \forall s < t, X_s(u) > x \right\},& x < 0.\\
  \end{cases}
\end{equation}
Jagers \cite{Jag} proved that branching processes stopped at $\mathcal{L}_x$ satisfies the branching property, i.e.\@ that each particle in $\mathcal{L}_x$ starts from its time and position an independent copy of the branching Brownian motion, which is independent of $\sigma\left((X_s(u), s \leq t), (u,t) \in \mathcal{L}_x\right)$.

We now associate to the branching Brownian motion the branching random walk of the birth places of particles, defined for all $u \in \mathcal{U}$ by $V(u) = X_u(d_u)$. Recall from \eqref{eqn:defineN} that our main quantity of interest can be written
\[
  N = \sum_{u \in \mathcal{U}} \ind{X_u(d_u) \leq 0} = \sum_{n \geq 0} \sum_{|u|=n} \ind{V(u) \leq 0},
\]
where the sum over $|u|=n$ is the sum over $u \in \{1,2\}^n$ the set of particles in the $n$th generation. From the construction of the branching random walk, it is apparent that $\displaystyle \left( V(\emptyset), V(u) - V(\pi u), u \in \mathcal{U} \backslash \{\emptyset\} \right)$ 
is a family of i.i.d.\@ random variables with same law as $\sqrt{2} B_T + 2T$, where  $\pi u$ is the parent of $u$, $B$ is a standard Brownian motion and $T$ an independent exponential random variable with parameter~$1$.

As a result, we deduce that $(V(u), u \in \mathcal{U})$ is a \emph{branching random walk}, a discrete-time particle system on the real line starting from $V(\emptyset)$, such that each parent particle gives birth to two daughter particles that are positioned around their parent according to i.i.d.\@ copies of $V(\emptyset)$. Observe that for all $\lambda$ close enough to $0$, we have
\begin{align*}
  \E(e^{\lambda X_{\emptyset}(d_\emptyset)}) &= \int_0^\infty \dd s e^{-s} \E\left(e^{\lambda (\sqrt{2} B_s + 2 s)}\right) = \int_0^\infty \dd s e^{-s \left(1- \lambda^2 -2 \lambda\right)}\\
  &= \frac{1}{1 - 2 \lambda - \lambda^2} = \frac{1}{2\sqrt{2}(\sqrt{2}+1 + \lambda)} + \frac{1}{2\sqrt{2}(\sqrt{2} - 1 - \lambda)}.
\end{align*}
Therefore the law of the displacement of the branching random walk $V$ has the density
$
  (1 - p)\ind{x < 0} (\sqrt{2}+1)e^{(\sqrt{2}+1)x} + p \ind{x > 0} (\sqrt{2} - 1)e^{-(\sqrt{2}-1)x}
$
with respect to the Lebesgue measure on $\R$, where $p := \frac{2+\sqrt{2}}{4} \approx 0.85355$.

For all $a \in \R$, we write $\P_a$ the law of $V$ conditionally on $V(\emptyset) = a$ and $\E_a$ the corresponding expectation. 
We next introduce the many-to-one lemma. This result has a long history going back to the work of Peyrière \cite{Pey} and Kahane and Peyrière \cite{KaP}. We refer to \cite[Theorem~1.1]{Shi} for a proof of this result.
\begin{lemma}[Many-to-one lemma]
\label{lem:mto}
For any $a \in \R$, $n \geq 1$ and measurable function $f: \R^n \to \R_+$, we have
\[
  \E_a\left( \sum_{|u|=n} f(V(u_1), \ldots V(u_n)) \right) = \E_a\left( e^{S_n-a} f(S_1,\ldots, S_n) \right),
\]
where $(u_1,\ldots,u_n)$ is the ancestral line of $u$ and $(S_n)_{n \geq 0}$ is a random walk such that $\P_a(S_0=a) = 1$, whose step distribution has density $\frac{\sqrt{2}}{2}e^{-\sqrt{2}|x|}$ with respect to the Lebesgue measure on $\R$.
\end{lemma}

As an immediate consequence of the above lemma, we obtain that
\begin{equation}
  \label{eqn:boundary}
  \E_0\left( \sum_{|u|=1} e^{-V(u)} \right) = 1 \quad \text{and} \quad \E_0\left( \sum_{|u|=1} V(u) e^{-V(u)} \right) = 0.
\end{equation}
Therefore, $V$ is a branching random walk in the boundary case, according to the terminology of \cite{BiK05}. We also note that 
\begin{equation}
  \label{eqn:variance}
  \E_0\left( \sum_{|u|=1} V(u)^2 e^{-V(u)} \right) = 1,
\end{equation}
i.e.\@ the step distribution of the random walk $(S_n)$ has unit variance.

We conclude this section with some random walk estimates. Stone's local limit theorem \cite{Sto} implies the existence of $C_0>0$ such that for all $n \in \N^*$,
\begin{equation}
  \label{eqn:localLimit}
  \sup_{z \in \R} \P_0\left( S_n \in [z,z+1] \right) \leq C_0 n^{-1/2}.
\end{equation}

For all $n \in \R$, we set $\underline{S}_n = \min_{k \leq n} S_k$. Estimates on the joint law of $S_n$ and $\underline{S}_n$ are often called ballot theorems in the literature, and often appear in the study of branching random walks. In this article, we need the three following ballot-type estimates: there exists $C_1 > 0$ such that for all $\alpha > 0$ and $n \in \N^*$,
\begin{equation}
  \label{eqn:ballot}
  \P_0\left( \underline{S}_n \geq -\alpha \right) \leq C_1 (1 + \alpha) n^{-1/2},
\end{equation}
there exists $C_2 > 0$ such that for any $h, \alpha > 0$,
\begin{equation}
  \label{eqn:ballotBackward}
  \P_0\left( \underline{S}_n \geq -\alpha, S_n \in [h -\alpha,h-\alpha +1] \right) \leq C_2 (1+h) n^{-1},
\end{equation}
and there exists $C_3>0$ such that for any $a > -\alpha$ and $h > 0$
\begin{equation}
  \label{eqn:estimate}
  \P_0\left( \underline{S}_n \geq -\alpha, S_n \in [a,a+h] \right) \leq C_3 (1 + \alpha)(1 +a + h + \alpha)(1+h) n^{-3/2}.
\end{equation}
These estimates can be obtained as immediate consequences of \cite[Lemma A.1]{AiSsimpleProof}.

\section{Proof of Theorem~\ref{thm:mainboundary}}
\label{sec:stoppingLine}

The proof of Theorem~\ref{thm:mainboundary} is based on the following decomposition of the number $N_x$ of birth events below level $x$ along the stopping line $\mathcal{L}_x$.
\begin{lemma}
\label{lem:stoppingLine}
Let $x \in \R$, we write $Z_x = \# \mathcal{L}_x$ for the total number of particles that hit level $x$ for the first time in their history. We have
\[
  N_x \egaldistr (Z_x-1)\ind{x>0} + \sum_{j =1}^{Z_x} N^{(j)},
\]
where $(N^{(j)}, j \geq 1)$ are i.i.d.\@ copies of $N$ which are further independent of $Z_x$.
\end{lemma}

This equality in distribution allows us to link together the law of $N$ with the laws of $N_x$ and $Z_x$. We then determine the asymptotic behaviour as $x \to \infty$ of $N_x$ and $Z_x$, and use those to obtain estimates on the law of $N$.

\begin{remark}
\label{rem:linkNZ}
As the branching Brownian motion splits at each time in exactly $2$ children at every branching event, and no particles stay forever below the level $x>0$, the total number of particles hitting level $x$ for the first time in their history satisfies $Z_x = \bar{N}_x + 1 < \infty$ a.s.\@, with $\bar{N}_x$ the total number of births given by particles before their absorption at level $x$, defined in \eqref{eqn:barN}.
\end{remark}

\begin{proof}
The above equality is an immediate consequence of the branching property applied at the stopping line $\mathcal{L}_x$. Each of the $Z_x$ particles of the stopping line starts an independent branching Brownian motion from level $x$, independently from the branching Brownian motion absorbed at level $x$.

As a consequence, the total number of births below level $x$ is equal to the number of births below level $x$ occurring before hitting $x$ for the first time (which is equal to $0$ if $x < 0$ or to $\bar{N}_x = Z_x-1$ if $x > 0$), summed with the total number of births below level $x$ of all the branching Brownian motions started from $\mathcal{L}_x$, which are equal in distribution to sum of $Z_x$ independent copies of $N_0$.
\end{proof}

As mentioned in Remark~\ref{rem:taub}, the proof of Theorem~\ref{thm:mainboundary} relies on a tight computation of the asymptotic behaviour of the Laplace transform of $N$. For all $x \in \R$ and $\lambda > 0$, we set
\[
  \phi(\lambda,x) = \log \E\left( e^{-\lambda N_x} \right).
\]
By Lemma~\ref{lem:stoppingLine}, we have
\begin{equation}
  \label{eqn:stoppingLine}
  \phi(\lambda,x) =
  \begin{cases}
    \lambda + \log \E\left( \exp\left( \left(\phi(\lambda,0) - \lambda\right)Z_x \right) \right)& \text{ if } x > 0\\
    \E\left( \exp\left( \phi(\lambda,0)Z_x \right) \right) &\text{ if } x < 0.
  \end{cases}
\end{equation}
In other words, the Laplace transform of $N$ can be related to the Laplace transform of $Z_x$ the number of particles first hitting level $x$, or equivalently by Remark~\ref{rem:linkNZ} to the Laplace transform of $\bar{N}_x$ when $x > 0$.

To study the asymptotic behaviour of $\phi(\lambda,0)$ as $\lambda \to 0$, we show that normalized versions of $Z_x$ and $N_x$ both converge, as $x \to \infty$, to multiples of the limit of the derivative martingale of the branching Brownian motion, defined in \eqref{eqn:dinfity}.

First, using stopping line techniques, we obtain an almost sure estimate for the growth rate of $Z_x$ as $x \to \infty$.
\begin{lemma}
\label{fct:asymptoticZ}
We have $\displaystyle \lim_{x \to \infty} xe^{- x} Z_x = D_\infty$ a.s.
\end{lemma}

\begin{proof}
This result is a direct consequence of \cite[Theorem~6.1]{BiK04}, stating that the derivative martingale stopped at line $\mathcal{L}_x$ converges, as $x \to \infty$ to $D_\infty$ a.s. In other words
$\lim_{x \to \infty} \sum_{(u,t) \in \mathcal{L}_x} X_t(u) e^{-X_t(u)} = \lim_{x \to \infty} x e^{-x} Z_x = D_\infty$ a.s.
\end{proof}

We now turn to the asymptotic behaviour, as $x \to \infty$, of
\[
  N_x = \sum_{u \in \mathcal{U}} \ind{X_u(d_u) \leq x} = \sum_{n \geq 0} \sum_{|u|=n} \ind{V(u) \leq x}.
\]
Using estimates developed in Chen \cite{Che20}, we are able to obtain the following asymptotic behaviour for $N_x$ as $x \to \infty$.
\begin{proposition}\label{prop:cvgN}
We have $\displaystyle \lim_{x \to \infty} e^{- x}N_x = 2 D_\infty$ in probability.
\end{proposition}

This convergence can be thought of as a Seneta-Heyde type result for the additive martingale of the branching random walk $V$. A similar convergence was obtained in \cite[Eq. (5.5)]{Che20}, using similar methods as the one pioneered by Boutaud and Maillard \cite{BM}. Precisely, for all $0 < a < b$ and $\Lambda >  1$, we have
\begin{multline}
  \label{eqn:che55}
  \lim_{x \to \infty} e^{-x}\sum_{a x^2 \leq n \leq bx^2} \sum_{|u|=n} \ind{V(u) \leq x, \max_{k \leq n} V(u_k) \leq \Lambda x}\\
   = \sqrt{\tfrac{2}{\pi}} D_\infty \int_a^b   \phi(u^{-1/2}) g(\Lambda u^{-1/2}, u^{-1/2})\frac{\dd u}{u} \quad \text{in probability},
\end{multline}
where $\phi(z) = z e^{-z^2/2}$ and $g(a,b) = \P(\sup_{s \in [0,1]} R_s \leq a | R_1 = b)$, with $R$ a Bessel process of dimension $3$. Here we used \cite[Lemma~2.1]{AiS14} to identify the constant $\sqrt{\tfrac{2}{\pi}}$.
To complete the proof of Proposition~\ref{prop:cvgN}, we use the following lemma, whose proof is postponed to the end of the article.

\begin{lemma}
\label{lem:cvgBounds}
For $u \in \mathcal{U}$, we set $\underline{V}(u) = \min_{k \leq |u|} V(u_k)$, $\bar{V}(u) = \max_{k \leq |u|} V(u_k)$.
For all $\alpha > 0$ and $0 < a < b$, we have
\begin{equation}
  \label{eqn:middlePart}
  \lim_{\Lambda \to \infty} \limsup_{x \to \infty} e^{-x} \E_0\left( \sum_{ax^2 \leq n \leq bx^2} \sum_{|u|=n} \ind{\underline{V}(u) \geq -\alpha, \bar{V}(u) \geq \Lambda x, V(u) \leq x} \right) = 0.
\end{equation}
Additionally, for all $\alpha > 0$,
\begin{align}
  \lim_{\delta \to 0} \limsup_{x \to \infty} e^{-x}\E_0\left[\sum_{n\leq \delta x^2}\sum_{|u|=n}\ind{V(u)\leq x,\underline{V}(u)\geq-\alpha}\right]=&0, \label{badsmallgeneration}\\
  \text{and}\quad \lim_{\delta \to 0} \limsup_{x \to \infty} e^{-x}\E_0\left[\sum_{n\geq  x^2/\delta}\sum_{|u|=n}\ind{V(u)\leq x,\underline{V}(u)\geq-\alpha}\right]=&0, \label{badlargegeneration}
\end{align}
\end{lemma}

\begin{proof}[Proof of Proposition~\ref{prop:cvgN}]
For all $0 \leq a < b \leq \infty$, we set
\[
  N_x(a,b) = \sum_{ax^2 \leq n \leq bx^2} \sum_{|u|=n} \ind{V(u) \leq x} \text{ and } c_{a,b}= \sqrt{\tfrac{2}{\pi}} \int_a^b \phi(u^{-1/2}) \frac{\dd u}{u}.
\]
Let $\epsilon > 0$, by \eqref{eqn:maxDep}, there exists $\alpha > 0$ such that $\P_0(\inf_{u \in \mathcal{U}} V(u) \leq -\alpha) < \epsilon$. Additionally, using \eqref{badsmallgeneration}, \eqref{badlargegeneration} and the Markov inequality, we can fix $A_1, \delta > 0$ such that for all $x \geq A_1$,
\[
  \P_0\left( N_x(0,\delta) + N_x(\delta^{-1},\infty) \geq \epsilon e^{x}\right) \leq \epsilon + \P(\inf_{u \in \mathcal{U}} V(u) \geq -\alpha) \leq 2\epsilon.
\]
Up to decreasing $\delta$, we assume as well that $0 \leq c_{0,\infty} - c_{\delta,\delta^{-1}} \leq \epsilon$. Similarly, using \eqref{eqn:middlePart}, we may fix $A_2 \geq A_1$ and $\Lambda > 1$ such that for all $x \geq A_2$, we have
\[
  \P_0\left( N_x(\delta,\delta^{-1}) - N_x(\delta,\delta^{-1},\Lambda) \geq \epsilon e^x \right) \leq 2 \epsilon,
\]
where $\displaystyle N_x(a,b,\Lambda)= \sum_{|u|=n} \ind{V(u) \leq x, \bar{V}(u) \leq \Lambda x}$. Up to enlarging again $\Lambda$, we also assume that
$\displaystyle  0 \leq c_{\delta,\delta^{-1}} - \sqrt{\tfrac{2}{\pi}} \int_{\delta}^{\delta^{-1}}   \phi(u^{-1/2}) g(\Lambda u^{-1/2}, u^{-1/2})\frac{\dd u}{u} \leq \epsilon.$
Finally, using the convergence \eqref{eqn:che55}, we can choose $A_3 \geq A_2$ such that for all $x \geq A_3$,
\[
  \P_0\left( \left| e^{-x}N_x(\delta,\delta^{-1},\Lambda) - c_{0,\infty} D_\infty \right| \geq (2D_\infty + 1) \epsilon \right) \leq \epsilon.
\]
As a result, for all $x \geq A_3$, chaining these equations we obtain
\[
  \P_0\left( \left| e^{-x}N_x - c_{0,\infty} D_\infty \right| \geq (3 + 2D_\infty) \epsilon \right) \leq 5 \epsilon,
\]
proving that $\lim_{x \to \infty} e^{-x} N_x = c_{0,\infty} D_\infty$ in $\P_0$-probability, as $D_\infty$ is a.s.\@ finite.

Next, using that $V(\emptyset)$ is independent of $(V(u)-V(\emptyset), u \in \mathcal{U})$, which has law $\P_0$, and that for all $a \in \R$, the law of $D_\infty$ under $\P_a$ is the same as the law of $e^a D_\infty$ under law $\P_0$, we also obtain that
\[
  \lim_{x \to \infty} e^{-x} N_x = c_{0,\infty} D_\infty \quad \text{ in $\P$-probability}.
\]
By computing that $c_{0,\infty} = \sqrt{\frac{2}{\pi}} \int_0^\infty \phi(u^{-1/2}) \frac{\dd u}{u} = 2$,  the proof is now compete.
\end{proof}

Next, using Lemma~\ref{fct:asymptoticZ} and Proposition~\ref{prop:cvgN}, we are now able to prove Theorem~\ref{thm:mainboundary}.
\begin{proof}[Proof of Theorem~\ref{thm:mainboundary}]
By Lemma~\ref{fct:asymptoticZ}, for all $\lambda > 0$ we have
\[
  \lim_{x \to \infty} \E\left( e^{-\lambda xe^{-x} Z_x} \right) = \E(e^{-\lambda D_\infty}).
\]
Recall that $\phi(\lambda,x) = \log \E(e^{-\lambda N_x})$, and as $N+1$ is a.s.\@ positive, the continuous function $\lambda \mapsto \lambda - \phi(\lambda,0)$ is increasing. Hence for all $x > 0$ large enough, there exists a unique $\lambda_x > 0$ so that
\[
  \lambda_x - \phi(\lambda_x,0)= x e^{- x},
\]
with $\lambda_x \to 0$ as $x \to \infty$. Then, equation \eqref{eqn:stoppingLine} implies that
\[
  \lim_{x \to \infty} \E(e^{-\lambda_x (N_x - 1)}) = \E(e^{-D_\infty}).
\]
On the other hand, it is known from Proposition~\ref{prop:cvgN} that for all $\mu > 0$,
\[
  \lim_{x \to \infty} \E(e^{-\mu e^{-  x} (N_x - 1)}) = \E(e^{-2 \mu D_\infty}).
\]

As a result, we conclude that $e^{- x} \sim 2 \lambda_x$ as $x \to \infty$. It yields in particular that $2 \phi(\lambda_x,0) = 2 (\lambda_x -x e^{-x}) = (-2 x + 1 +o(1))e^{-x}$ as $x \to \infty$, i.e.\@ as $\lambda_x \to 0$. As a result, we obtain that as $\lambda \to 0$, we have
\begin{align*}
  \phi(\lambda,0) &= \lambda (\log (2 \lambda) + 1 + o(1))\\
  &= \lambda \log \lambda + (1 + \log 2) \lambda + o(\lambda) \quad \text{as $\lambda \to 0$}.
\end{align*}
By Remark~\ref{rem:taub}, we deduce
\[
  \E\left( N \ind{N \leq n} \right) = \log n  + (\log 2 + \gamma) + o(1),
\]
as $n \to \infty$, which completes the proof of the main theorem.
\end{proof}

We end this article with a proof of Lemma~\ref{lem:cvgBounds}, which is based on the many-to-one lemma and random walk estimates.
\begin{proof}[Proof of Lemma~\ref{lem:cvgBounds}]
We prove each of the three limits in turn, using the ballot-type random walk estimates introduced in Section~\ref{sec:stoppingLine}.

\textbf{Proof of \eqref{eqn:middlePart}.} For all $n \in \N$, we set $\bar{S}_n = \max_{k \leq n} S_k$. Let $0 < a < b$, using the many-to-one lemma, we compute for all $\alpha, x > 0$ and $\Lambda>1$,
\begin{align}
  &\E_0\left( \sum_{a x^2 \leq n \leq bx^2} \sum_{|u|=n} \ind{\underline{V}(u) \geq -\alpha, \bar{V}(u) \geq \Lambda x, V(u) \leq x} \right)\nonumber\\
  = &\sum_{a x^2 \leq n \leq bx^2} \E_0\left( e^{S_n} \ind{\underline{S}_n \geq -\alpha, \bar{S}_n \geq \Lambda x, S_n \leq x} \right)\nonumber\\
  \leq &\sum_{ax^2\leq n \leq bx^2} \sum_{j = 0}^{\infty} e^{x - j} \P_0\left(\underline{S}_n \geq -\alpha, \bar{S}_n \geq \Lambda x, S_n \in [x - j-1, x - j] \right)\nonumber\\
  \leq &\frac{e^x}{1 - e^{-1}} (bx^2 - ax^2 + 1)\sup_{\substack{n \geq ax^2,\\h \leq (n/a)^{1/2}}} \P_0\left( \underline{S}_n \geq -\alpha, \bar{S}_n \geq \Lambda (n/b)^{1/2}, S_n \in [h-1,h] \right).\label{eqn:ub}
\end{align}
We then bound $\P_0\left( \underline{S}_n \geq -\alpha, \bar{S}_n \geq \Lambda (n/b)^{1/2}, S_n \in [h-1,h] \right)$ for large values of $\Lambda$, uniformly in $h \leq (n/a)^{1/2}$. 

Write $T^{(n)} = \inf\{ k \in \N : S_k \geq \Lambda (n/b)^{1/2}\}$, we observe that, setting $p = \floor{n/2}$,
\begin{multline}
  \label{eqn:splitTime}
\P_0\left( \underline{S}_n \geq -\alpha, \bar{S}_n \geq \Lambda (n/b)^{1/2}, S_n \in [h-1,h] \right)\\
\leq \P_0\left( \underline{S}_n \geq -\alpha, T^{(n)} \leq p, S_n \in [h-1,h] \right)\\ + \P_0\left( \underline{S}_n \geq -\alpha, T^{(n)}\in ]p,n], S_n \in [h-1,h]  \right),
\end{multline}
and we bound these two probabilities in turn. Applying the Markov property at time $p$, we have
\begin{multline*}
  \P_0\left( \underline{S}_n \geq -\alpha, T^{(n)} \leq p, S_n \in [h-1,h] \right)\\
   \leq \P_0\left(\underline{S}_{p} \geq - \alpha, \bar{S}_{p} \geq \Lambda(n/b)^{1/2}\right)  \sup_{z \in \R} \P(S_{n- p} \in [z-1,z])\\
   \leq \frac{C_0 C_1 (1 + \alpha)}{p^{1/2} (n-p)^{1/2}} \P_0\left( \bar{S}_p \geq \Lambda (n/b)^{1/2} \middle| \underline{S}_p \geq -\alpha  \right),
\end{multline*}
using \eqref{eqn:localLimit} and \eqref{eqn:ballot}. Therefore, there exists $C'>0$ such that for all $n \in \N$, we have
\begin{multline*}
  \P_0\left( \underline{S}_n \geq -\alpha, \bar{S}_p \geq \Lambda (n/b)^{1/2}, S_n \in [h-1,h] \right)\\ \leq \frac{C'}{n} \P_0\left( \bar{S}_p \geq \Lambda (n/b)^{1/2} \middle| \underline{S}_p \geq -\alpha  \right).
\end{multline*}
Given $R$ a Bessel process of dimension 3, by Caravenna-Chaumont's invariance principle \cite[Theorem~1.1]{CaC08}, we have
\[
  \lim_{n \to \infty} \P_0\left( \bar{S}_p \geq \Lambda (n/b)^{1/2} \middle| \underline{S}_p \geq -\alpha  \right) = \P(\max_{s \in [0,1]} R_s \geq \Lambda (2/b)^{1/2}),
\]
which converges to $0$ as $\Lambda \to \infty$. Hence, we conclude that
\begin{equation}
 \label{eqn:firstPart}
 \lim_{\Lambda \to \infty} \limsup_{n \to \infty} \left[ n \sup_{h \leq (n/a)^{1/2}} \P_0\left( \underline{S}_n \geq -\alpha, T^{(n)} \leq p, S_n \in [h-1,h] \right)\right] = 0.
\end{equation}

Next, observing that $(S_{n} - S_{n-k}, k \leq n) \egaldistr (S_k, k \leq n)$ by reversing time, for all $0 \leq h \leq (n/a)^{1/2}$, we have
\begin{align*}
  &\P_0\left( \underline{S}_n \geq -\alpha, \exists k \in ]p,n] : S_k \geq \Lambda (n/b)^{1/2}, S_n \in [h-1,h]  \right)\\
  \leq &\P_0\left( \max_{n-p \leq j \leq n} S_j \leq \alpha + h, \underline{S}_{n-p} \leq -\Lambda (n/b)^{1/2} + h, S_n \in [h-1,h] \right)\\
  \leq &\P_0\left( \underline{S}_{n-p} \leq -\Lambda(n/b)^{1/2} + (n/a)^{1/2} \right) \sup_{z \in \R} \P\left( \bar{S}_p \leq \alpha + z, S_{p} \in [z-1,z]  \right),
\end{align*}
applying the Markov property at time $n-p$. Then applying \eqref{eqn:ballotBackward} to the random walk $-S$, there exists $C'' > 0$ such that
\begin{multline*}
  \sup_{z \in \R} \P\left( \bar{S}_p \leq \alpha + z, S_{p} \in [z-1,z]  \right)\\
  = \sup_{z \in \R} \P\left( \underline{-S}_p \geq -\alpha - z, -S_p \in [-z,-z+1] \right) \leq \frac{C''(1+\alpha)}{n},
\end{multline*}
and by Donsker's invariance principle,
\[
  \lim_{n \to \infty} P_0\left( \underline{S}_{n-p} \leq -\Lambda(n/b)^{1/2} + (n/a)^{1/2} \right) = \P(\inf_{s \in [0,1]} B_s \leq - \Lambda(2/b)^{1/2} + (2/a)^{1/2}),
\]
where $B$ is a Brownian motion. As a result, we deduce that
\begin{equation}
 \label{eqn:secondPart}
 \lim_{\Lambda \to \infty} \limsup_{n \to \infty} \left[n \sup_{h \leq (n/a)^{1/2}} \P_0\left( \underline{S}_n \geq -\alpha, T^{(n)} \in ]p,n], S_n \in [h-1,h] \right)\right] = 0.
\end{equation}

Then, plugging \eqref{eqn:firstPart} and \eqref{eqn:secondPart} into \eqref{eqn:splitTime}, we deduce that
\[
  \lim_{\Lambda \to \infty} \limsup_{x \to \infty} \left[x^2 \sup_{\substack{n \geq ax^2,\\h \leq (n/a)^{1/2}}} \P_0\left( \underline{S}_n \geq -\alpha, \bar{S}_n \geq \Lambda (n/b)^{1/2}, S_n \in [h-1,h] \right)\right] = 0,
\]
which, going back to \eqref{eqn:ub} is enough to prove \eqref{eqn:middlePart}.

\textbf{Proof of \eqref{badlargegeneration}.} Using the many-to-one lemma, we have
\begin{align*}
&\E_0\left[\sum_{n\geq  x^2/\delta}\sum_{|u|=n}\ind{V(u)\leq x,\underline{V}(u)\geq-\alpha}\right]=\sum_{n\geq x^2/\delta}\E_0[e^{S_n}; S_n\leq x, \underline{S}_n\geq-\alpha]\\
\leq & \sum_{n\geq x^2/\delta}\sum_{k = \floor{-\alpha}}^\floor{x} e^{k+1}\P_0(\underline{S}_n\geq-\alpha, S_n\in[k,k+1])\\
\leq &  \sum_{n\geq x^2/\delta}\frac{C(1+\alpha)}{n^{3/2}} \sum_{k=0}^{x+\alpha} e^{k-\alpha}(1+k)
\end{align*}
by \eqref{eqn:estimate}. As a result, we obtain that
\[
  e^{-x} \E_0\left[\sum_{n\geq  x^2/\delta}\sum_{|u|=n}\ind{V(u)\leq x,\underline{V}(u)\geq-\alpha}\right] \leq C (1+\alpha)(1+x+\alpha) \frac{\sqrt{\delta}}{x}
\]
which is $o_\delta(1)(1+\alpha)$ as $\delta\downarrow 0$, uniformly in $x \geq 1$.

\textbf{Proof of \eqref{badsmallgeneration}.}
This proof is similar to the proof of \eqref{badlargegeneration}, using the same lines as in the proof of \cite[Equation (A.19)]{Che20}. We first use the many-to-one lemma to write 
\begin{align}
  \E_0\left[\sum_{n\leq \delta x^2}\sum_{|u|=n}\ind{V(u)\leq x,\underline{V}(u)\geq-\alpha}\right]&= \sum_{n\leq \delta x^2}\E_0[e^{S_n}; S_n\leq x,\underline{S}_n\geq-\alpha]\nonumber\\
  &\leq \sum_{n=1}^{\delta x^2} \sum_{r=\floor{-\alpha}}^{\floor{x}} e^{r+1}\P_0(S_n\in[r,r+1], \underline{S}_n\geq-\alpha), \label{eqn:formula2}
\end{align}
and we bound $\P_0(S_n \in [r,r+1], \underline{S}_n \geq -\alpha)$ uniformly in $n \leq \delta x^2$ and $r \leq x$. To do so, we recall the following inequality, first proved in \cite{Andreoletti-Roland} for random walks with bounded increments: as $\E[e^{\eta|S_1|}]<\infty$ for some $\eta>0$, there exist $0<a<1<b<\infty$ such that for  $r\geq 1$ and $br\leq n\leq ar^2$, 
\[
\P_0(\underline{S}_n\geq-\alpha, S_n\in[r,r+1])\leq C\frac{1+\alpha}{n}e^{-c'\frac{r^2}{n}}
\]
It comes from the fact that $\P_0(S_k\in[r,r+1])\leq \frac{1}{\sqrt{k}}e^{-cr^2/k}$ for all $br\leq k\leq ar^2$. On the other hand, if $1\leq n\leq br$ and $0<t<\eta$, one has
\[
\sum_{n=1}^{br}\P_0(\underline{S}_n\geq-\alpha, S_n\in[r,r+1])\leq \sum_{n=1}^{br}\P(S_n\geq r)\leq e^{-t r}\sum_{n=1}^{br}\E[e^{tS_1}]^n\leq e^{-c r}.
\]
Consequently,
\[
\sum_{n=1}^{\delta r^2}\P_0(\underline{S}_n\geq-\alpha, S_n\in[r,r+1])\leq c(1+\alpha)\delta,
\]
which by \eqref{eqn:formula2} yields
$\displaystyle
  e^{-x}\E_0\left[\sum_{n\leq \delta x^2}\sum_{|u|=n}\ind{V(u)\leq x,\underline{V}(u)\geq-\alpha}\right] \leq C \delta (1+\alpha),
$
completing the proof of \eqref{badsmallgeneration}.
\end{proof}

\paragraph{Acknowledgements.}
We are grateful to the anonymous referee for the relevant comments on the first version of this article. The authors were partially supported by ANR grant MALIN
(ANR-16-CE93-0003) during the redaction of this article.


\begin{thebibliography}{BBHM17}

\bibitem[AB11]{ABB}
L.~{Addario-Berry} and N.~{Broutin},
  \emph{{Total progeny in
  killed branching random walk.}}, {Probab. Theory Relat. Fields} \textbf{151}
  (2011), no.~1-2, 265--295.

\bibitem[AHZ13]{ahz13}
E.~A\"{i}d\'{e}kon, Y.~Hu, and O.~Zindy,
  \emph{{The precise tail behavior of
  the total progeny of a killed branching random walk}}, Ann. Probab.
  \textbf{41} (2013), no.~6, 3786--3878.

\bibitem[AS10]{AiSsimpleProof}
Elie {A\"{\i}d\'ekon} and Zhan {Shi}, \emph{{Weak convergence for the minimal
  position in a branching random walk: a simple proof}}, {Period. Math. Hung.}
  \textbf{61} (2010), no.~1-2, 43--54.

\bibitem[AD08]{Andreoletti-Roland}
P.~Andreoletti and R.~Diel,
  \emph{{Limit Law of the Local
  Time for Brox’s Diffusion}}, Journal of Theoretical Probability \textbf{24}
  (2008), 634--656.

\bibitem[Aïd10]{Aid}
E.~{Aïd\'{e}kon}, \emph{{Tail
  asymptotics for the total progeny of the critical killed branching random
  walk.}}, {Electron. Commun. Probab.} \textbf{15} (2010), 522--533.

\bibitem[AS14]{AiS14}
E.~{A{ï}d\'{e}kon} and Z.~{Shi}, \emph{{The Seneta-Heyde scaling for the branching
  random walk}}, {Ann. Probab.} \textbf{42} (2014), no.~3, 959--993.

\bibitem[BBHM17]{BBHM15}
J.~{Berestycki}, \'E. {Brunet}, S.~C. {Harris}, and P.~{Mi{\l}o\'s},
  \emph{{Branching Brownian
  motion with absorption and the all-time minimum of branching Brownian motion
  with drift.}}, {J. Funct. Anal.} \textbf{273} (2017), no.~6, 2107--2143.

\bibitem[BK04]{BiK04}
J.~D. {Biggins} and A.~E. {Kyprianou},
  \emph{{Measure change in multitype
  branching.}}, {Adv. Appl. Probab.} \textbf{36} (2004), no.~2, 544--581.

\bibitem[BK05]{BiK05}
J.~D. {Biggins} and A.~E. {Kyprianou},
  \emph{{Fixed points of the
  smoothing transform: the boundary case.}}, {Electron. J. Probab.} \textbf{10}
  (2005), 609--631, Id/No 17.

\bibitem[BM19]{BM}
P.~{Boutaud} and P.~{Maillard},
  \emph{{A revisited proof of the
  Seneta-Heyde norming for branching random walks under optimal assumptions.}},
  {Electron. J. Probab.} \textbf{24} (2019), 22, Id/No 99.

\bibitem[Bra78]{Bra78}
M.~D. Bramson,
  \emph{{Maximal
  displacement of branching brownian motion}}, Comm. Pure Appl. Math.
  \textbf{31} (1978), no.~5, 531--581.

\bibitem[BIM20]{BIM20+}
D.~Buraczewski, A.~Iksanov, and B.~Mallein,
  \emph{{On the derivative martingale in
  a branching random walk}}, arXiv:2002.05215, submitted, feb 2020.

\bibitem[CC08]{CaC08}
F.~{Caravenna} and L.~{Chaumont},
  \emph{{Invariance principles for
  random walks conditioned to stay positive}}, {Ann. Inst. Henri Poincar\'e,
  Probab. Stat.} \textbf{44} (2008), no.~1, 170--190.

\bibitem[Cha86]{Cha}
B.~{Chauvin},
  \emph{{Arbres et
  processus de Bellman-Harris. (Trees and Bellman-Harris processes).}}, {Ann.
  Inst. Henri Poincar\'e, Probab. Stat.} \textbf{22} (1986), 209--232.

\bibitem[Che20]{Che20}
Xinxin Chen, \emph{{Heavy range of the
  randomly biased walk on Galton-Watson trees in the slow movement regime}},
  arXiv:2009.13866, submitted, sep 2020.

\bibitem[Jag89]{Jag}
P.~{Jagers}, \emph{{General
  branching processes as Markov fields.}}, {Stochastic Processes Appl.}
  \textbf{32} (1989), no.~2, 183--212.

\bibitem[KP76]{KaP}
J.-P. {Kahane} and J.~{Peyriere},
  \emph{{Sur certaines
  martingales de Benoit Mandelbrot}}, {Adv. Math.} \textbf{22} (1976), 131--145.

\bibitem[Kes78]{Kes78}
H.~{Kesten},
  \emph{{Branching Brownian
  motion with absorption.}}, {Stochastic Processes Appl.} \textbf{7} (1978),
  9--47.

\bibitem[LS87]{LS87}
S.~P. Lalley and T.~Sellke,
  \emph{{A Conditional Limit
  Theorem for the Frontier of a Branching Brownian Motion}}, Ann. Probab.
  \textbf{15} (1987), no.~3, 1052--1061.

\bibitem[Mai13]{Mai13}
P.~{Maillard}, \emph{{The number of
  absorbed individuals in branching Brownian motion with a barrier.}}, {Ann.
  Inst. Henri Poincar\'e, Probab. Stat.} \textbf{49} (2013), no.~2, 428--455.

\bibitem[Pey74]{Pey}
Jacques {Peyriere}, \emph{{Turbulence et dimension de Hausdorff}}, {C. R. Acad.
  Sci., Paris, S\'er. A} \textbf{278} (1974), 567--569.

\bibitem[Shi15]{Shi}
Zhan {Shi}, \emph{{Branching random walks. \'Ecole d'\'Et\'e de Probabilit\'es
  de Saint-Flour XLII -- 2012}}, vol. 2151, Cham: Springer, 2015.

\bibitem[Sto67]{Sto}
Charles {Stone}, \emph{{On local and
  ratio limit theorems}}, {Proc. 5th Berkeley Sympos. math. Statist. Probab.,
  Univ. Calif. 1965/1966, 2, Part 2 (1967), 217-224}.

\end{thebibliography}
\end{document}